\ifdefined\pdfoutput
  \pdfoutput=1
\fi

\documentclass[11pt,a4paper]{article}

\usepackage[T1]{fontenc}
\usepackage[utf8]{inputenc}
\usepackage[english]{babel}
\usepackage{lmodern}
\usepackage[final,protrusion=true,expansion=true]{microtype}
\usepackage[
  a4paper,
  top=2.25cm,
  bottom=2.25cm,
  left=2.65cm,
  right=2.65cm,
  footskip=28pt,
  marginparwidth=1.75cm
]{geometry}

\usepackage{setspace}
\usepackage{amsmath,amssymb,amsfonts,amsthm,mathtools}
\usepackage{etoolbox}
\usepackage{bm}
\usepackage{bbm}
\allowdisplaybreaks
\numberwithin{equation}{section}
\mathtoolsset{showonlyrefs=false}

\usepackage{xcolor}
\usepackage{graphicx}
\graphicspath{{figures/}{images/}}
\usepackage{booktabs}
\usepackage{array}
\usepackage{adjustbox}
\usepackage{caption}
\usepackage{subcaption}
\usepackage{float}
\usepackage[section]{placeins}

\usepackage{tikz}
\usetikzlibrary{
  arrows.meta,
  calc,
  decorations.markings,
  decorations.pathreplacing,
  positioning
}
\tikzset{
  vtx/.style={
    circle,
    fill=black,
    draw=black,
    inner sep=2pt
  },
  gluevtx/.style={
    circle,
    fill=red,
    draw=red,
    inner sep=2.15pt
  },
  ed/.style={
    blue,
    line width=1.15pt
  },
  every picture/.style={
    line cap=round,
    line join=round
  }
}

\usepackage{algorithm}
\usepackage{algpseudocode}
\floatname{algorithm}{Algorithm}

\usepackage{enumitem}
\setlist[itemize]{
  leftmargin=2.2em,
  itemsep=0.15em,
  topsep=0.35em
}
\setlist[enumerate]{
  leftmargin=2.2em,
  itemsep=0.15em,
  topsep=0.35em
}

\usepackage{authblk}

\usepackage{titlesec}

\titleformat{\section}
  {\large\bfseries}
  {\thesection.}
  {0.65em}
  {}

\titleformat{\subsection}
  {\normalsize\bfseries}
  {\thesubsection.}
  {0.65em}
  {}

\titleformat{\subsubsection}
  {\normalsize\itshape}
  {\thesubsubsection.}
  {0.65em}
  {}

\titlespacing*{\section}
  {0pt}
  {2.4ex plus .4ex minus .2ex}
  {1.1ex}

\titlespacing*{\subsection}
  {0pt}
  {1.9ex plus .3ex minus .2ex}
  {0.8ex}

\titlespacing*{\subsubsection}
  {0pt}
  {1.5ex plus .2ex minus .2ex}
  {0.6ex}

\usepackage{aliascnt}

\theoremstyle{plain}

\newtheorem{theorem}{Theorem}[section]

\newaliascnt{lemma}{theorem}
\newtheorem{lemma}[lemma]{Lemma}
\aliascntresetthe{lemma}

\newaliascnt{question}{theorem}

\aliascntresetthe{question}

\newaliascnt{proposition}{theorem}

\aliascntresetthe{proposition}

\newaliascnt{corollary}{theorem}
\newtheorem{corollary}[corollary]{Corollary}
\aliascntresetthe{corollary}

\newaliascnt{conjecture}{theorem}
\newtheorem{conjecture}[conjecture]{Conjecture}
\aliascntresetthe{conjecture}

\newaliascnt{problem}{theorem}
\newtheorem{problem}[problem]{Problem}
\aliascntresetthe{problem}

\newaliascnt{observation}{theorem}

\aliascntresetthe{observation}

\newenvironment{subproof}[1][Proof of the claim]{%
  \par\noindent\textit{#1. }%
}{\hfill$\Box$\par}

\newcounter{proofgroup}

\newtheoremstyle{indentedplain}
  {\topsep}       
  {\topsep}       
  {\itshape}      
  {\parindent}    
  {\bfseries}     
  {.}             
  {0.5em}         
  {}              

\theoremstyle{indentedplain}

\newtheorem{claim}{Claim}[proofgroup]

\AtBeginEnvironment{proof}{\stepcounter{proofgroup}}

\theoremstyle{definition}

\newaliascnt{definition}{theorem}

\aliascntresetthe{definition}

\newaliascnt{example}{theorem}

\aliascntresetthe{example}

\theoremstyle{remark}

\newaliascnt{remark}{theorem}

\aliascntresetthe{remark}

\newaliascnt{fact}{theorem}

\aliascntresetthe{fact}

\makeatletter
\expandafter\patchcmd\csname\string\proof\endcsname
  {\itshape}
  {\bfseries\upshape}
  {}
  {\PackageWarning{template}{Could not make the proof heading bold}}
\makeatother

\usepackage[numbers,sort&compress]{natbib}
\definecolor{linkblue}{RGB}{0,70,150}
\definecolor{citepurple}{RGB}{82,20,160}

\usepackage[
  colorlinks=true,
  linkcolor=linkblue,
  citecolor=citepurple,
  urlcolor=black,
  bookmarksnumbered=true,
  bookmarksopen=true
]{hyperref}

\usepackage{bookmark}
\usepackage[nameinlink,noabbrev,capitalise]{cleveref}

\crefname{theorem}{Theorem}{Theorems}
\crefname{lemma}{Lemma}{Lemmas}
\crefname{proposition}{Proposition}{Propositions}
\crefname{corollary}{Corollary}{Corollaries}
\crefname{conjecture}{Conjecture}{Conjectures}
\crefname{problem}{Problem}{Problems}
\crefname{observation}{Observation}{Observations}
\crefname{claim}{Claim}{Claims}
\crefname{case}{Case}{Cases}
\crefname{definition}{Definition}{Definitions}
\crefname{example}{Example}{Examples}
\crefname{remark}{Remark}{Remarks}
\crefname{question}{Question}{Questions}
\crefname{fact}{Fact}{Facts}
\crefname{equation}{equation}{equations}
\crefname{figure}{Figure}{Figures}
\crefname{table}{Table}{Tables}
\crefname{algorithm}{Algorithm}{Algorithms}

\Crefname{equation}{Equation}{Equations}

\title{
  \bfseries\boldmath
  Optimal coloring of
  $\{\mathrm{cap},\mathrm{even\ hole}\}$-free graphs
  with no short odd holes
}

\author{
  Feng Liu\thanks{
    Email:
    \href{mailto:liufeng0609@126.com}
    {liufeng0609@126.com}.
  }
}

\author{
  Shuang Sun\thanks{
    Email:
    \href{mailto:chocolatesun@sjtu.edu.cn}
    {chocolatesun@sjtu.edu.cn}~(corresponding author).
  }
}

\author{
  Yan Wang\thanks{
    Email:
    \href{mailto:yan.w@sjtu.edu.cn}
    {yan.w@sjtu.edu.cn}.
  }
}

\affil{
  School of Mathematical Sciences,
  Shanghai Jiao Tong University,
  Shanghai 200240, China
}

\date{}

\begin{document}

\maketitle

\begin{abstract}
A \emph{hole} is an induced cycle of length at least four, and an
\emph{even hole} is a hole of even length. A \emph{cap} is obtained
from a hole by adding a vertex adjacent to exactly two consecutive
vertices of the hole. Chen, Xu, and Xu proved that every
$\{\mathrm{cap},\mathrm{even\ hole}\}$-free graph $G$ satisfies
$\chi(G)\leq \left\lceil\frac{5}{4}\omega(G)\right\rceil$, and improved
this bound to
$\chi(G)\leq \left\lceil\frac{7}{6}\omega(G)\right\rceil$ when
$5$-holes are also excluded. They asked whether, for every integer
$q\geq3$, every $\{\mathrm{cap},\mathrm{even\ hole}\}$-free graph $G$
with no odd hole of length at most $2q-1$ satisfies
$$
\chi(G)\leq
\left\lceil\frac{2q+1}{2q}\omega(G)\right\rceil.
$$
We answer this question affirmatively and show that the bound is sharp
for every $q\geq3$.

\smallskip


\end{abstract}

\section{Introduction}

All graphs considered in this paper are finite and simple. We use standard graph-theoretic terminology and notation; see \cite{Bondy2008,West1996}. For a positive integer $k$, let $[k]=\{1,\ldots,k\}$. A proper $k$-coloring of a graph $G$ is a map $\phi\colon V(G)\to [k]$ such that $\phi(u)\ne\phi(v)$ whenever $uv\in E(G)$. The chromatic number $\chi(G)$ is the least integer $k$ for which $G$ has a proper $k$-coloring, and the clique number $\omega(G)$ is the maximum size of a clique in $G$.

For a graph $H$, a graph $G$ is \emph{$H$-free} if it contains no induced subgraph isomorphic to $H$. More generally, for a family $\mathcal F$ of graphs, $G$ is \emph{$\mathcal F$-free} if it is $F$-free for every $F\in\mathcal F$. A hereditary class $\mathcal G$ is \emph{$\chi$-bounded} if there exists a function $f$ such that $\chi(G)\leq f(\omega(G))$ for every $G\in\mathcal G$. Such a function is called a \emph{$\chi$-binding function} for $\mathcal G$. A central problem in this area is to determine which classes of $\mathcal F$-free graphs are $\chi$-bounded and, when they are, to find the best possible $\chi$-binding function. This problem was initiated by Gy\'arf\'as~\cite{Gya75}; see also the survey of Scott and Seymour~\cite{ScottSeymourSurvey}.

Perfect graphs are a classical example of a $\chi$-bounded class. A graph $G$ is \emph{perfect} if $\chi(H)=\omega(H)$ for every induced subgraph $H$ of $G$. A \emph{hole} is an induced cycle of length at least four, and it is \emph{even} or \emph{odd} according to its length. An \emph{antihole} is the complement of a hole. The Strong Perfect Graph Theorem due to Chudnovsky, Robertson, Seymour, and Thomas~\cite{SPGT} states that a graph is perfect if and only if it contains no odd hole and no odd antihole. This characterization places holes at the center of the study of $\chi$-boundedness.

A classical family of graphs obtained by excluding holes is the class of hole-free graphs, also known as chordal graphs. Dirac~\cite{Dir61} proved that every chordal graph has a vertex whose neighbours are pairwise adjacent. Applying this result to every induced subgraph shows that chordal graphs are perfect. Consequently, every hole-free graph $G$ satisfies $\chi(G)=\omega(G)$.

Substantial progress has also been made for classes in which only certain holes are excluded. Scott and Seymour~\cite{SS16} proved that every odd-hole-free graph $G$ satisfies $\chi(G)\leq \frac{2^{2^{\omega(G)+1}}}{48(\omega(G)+1)}$. More generally, Chudnovsky, Scott, and Seymour~\cite{CSS17} proved that, for every integer $\ell\geq 4$, the class of graphs with no hole of length at least $\ell$ is $\chi$-bounded. Chudnovsky, Scott, Seymour, and Spirkl~\cite{CSSS20} later proved that the class of graphs with no odd hole of length at least $\ell$ is also $\chi$-bounded. These results settled three long-standing conjectures of Gy\'arf\'as~\cite{Gya87}.

Another natural restriction is to require all holes to have the same length. For an integer $\ell\geq 4$, a graph is \emph{$\ell$-holed} if every hole in the graph has length $\ell$. There has been study of 4-holed graphs. 
Sivaraman~\cite{Sivaraman2018} proved that if every hole of a graph $G$ has length $4$, then $\chi(G)\le 2^{2^{\omega(G)}}$.
He also
conjectured that every 4-holed graph $G$ satisfies
$\chi(G)\le \omega(G)^2$, and recorded an improvement of Seymour
to the bound $\chi(G)\le 2^{\omega(G)^2}$.
Cook, Horsfield, Preissmann, Robin, Seymour,
Sintiari, Trotignon, and Vu\v{s}kovi\'c~\cite{CookHPRSSTV24} gave a structural characterization
of $\ell$-holed graphs for every $\ell\ge7$.
Wang and Wu~\cite{WW25} proved that, for every odd integer $\ell\geq 7$, every $\ell$-holed graph $G$ satisfies $\chi(G)\leq \left\lceil\frac{\ell}{\ell-1}\omega(G)\right\rceil$.

There have also been much study on even-hole-free graphs. In contrast to the double-exponential bound known for odd-hole-free graphs, even-hole-free graphs admit a linear $\chi$-binding function. Chudnovsky and Seymour~\cite{CS23} proved that every even-hole-free graph has a vertex whose neighborhood can be covered by two cliques. It follows that every even-hole-free graph $G$ satisfies $\chi(G)\leq 2\omega(G)-1$. However, whether this bound is best possible remains open. Huang, Zhou, and Chang~\cite{HZC26} proposed the following conjectural bound.

\begin{conjecture}[Huang--Zhou--Chang~\cite{HZC26}]
Every even-hole-free graph $G$ satisfies $\chi(G)\leq \left\lceil\frac{5}{4}\omega(G)\right\rceil$.
\end{conjecture}

Several subclasses of even-hole-free graphs have been studied. 
A \emph{pan} is obtained from a hole by adding a pendant edge; see Figure~\ref{fig:specialgraphs}. 
Cameron, Chaplick, and Ho\`ang~\cite{CCH18} proved that every $\{\mathrm{pan},\mathrm{even\ hole}\}$-free graph $G$ satisfies $\chi(G)\leq \left\lceil\frac{3}{2}\omega(G)\right\rceil$. A \emph{diamond} is obtained from $K_4$ by deleting one edge; see Figure~\ref{fig:specialgraphs}. Kloks, M\"uller, and Vu\v{s}kovi\'c~\cite{KMV09} proved that every $\{\mathrm{diamond},\mathrm{even\ hole}\}$-free graph $G$ satisfies $\chi(G)\leq \omega(G)+1$.
For a positive integer $t$, let $P_t$ denote the path on $t$ vertices. Karthick and Maffray~\cite{KM19} proved that every $\{P_6,\mathrm{even\ hole}\}$-free graph $G$ satisfies $\chi(G)\leq \left\lceil\frac{5}{4}\omega(G)\right\rceil$. More recently, Huang, Zhou, and Chang~\cite{HZC26} extended this result by proving the same bound for $\{P_7,\mathrm{even\ hole}\}$-free graphs. For more background on even-hole-free graphs, we refer to the survey of Vu\v{s}kovi\'c~\cite{Vus10}.

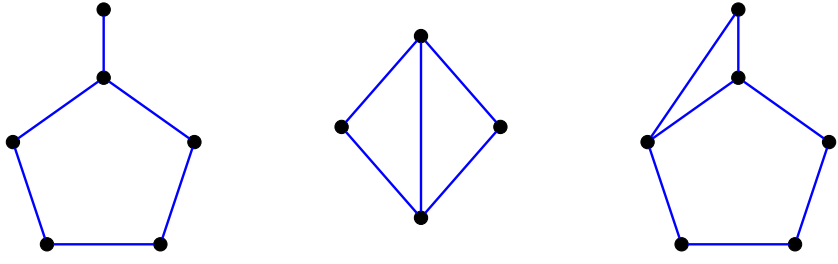
\begin{figure}[htbp]
    \centering
    \tikzstyle{v}=[circle, draw, fill=black, inner sep=0pt, minimum size=5pt]

    \begin{tikzpicture}[scale=1.0, baseline=(current bounding box.center)]
        \node[v] (a1) at (0,1.25) {};
        \node[v] (a2) at (1.2,0.4) {};
        \node[v] (a3) at (0.75,-0.95) {};
        \node[v] (a4) at (-0.75,-0.95) {};
        \node[v] (a5) at (-1.2,0.4) {};
        \node[v] (a6) at (0,2.15) {};
        \draw[blue, line width=0.9pt] (a1)--(a2)--(a3)--(a4)--(a5)--(a1);
        \draw[blue, line width=0.9pt] (a1)--(a6);
    \end{tikzpicture}
    \hspace{1.5cm}
    \begin{tikzpicture}[scale=1.0, baseline=(current bounding box.center)]
        \node[v] (c1) at (0,1.2) {};
        \node[v] (c2) at (1.05,0) {};
        \node[v] (c3) at (0,-1.2) {};
        \node[v] (c4) at (-1.05,0) {};
        \draw[blue, line width=0.9pt] (c1)--(c2)--(c3)--(c4)--(c1);
        \draw[blue, line width=0.9pt] (c1)--(c3);
    \end{tikzpicture}
    \hspace{1.5cm}
    \begin{tikzpicture}[scale=1.0, baseline=(current bounding box.center)]
        \node[v] (b1) at (0,1.25) {};
        \node[v] (b2) at (1.2,0.4) {};
        \node[v] (b3) at (0.75,-0.95) {};
        \node[v] (b4) at (-0.75,-0.95) {};
        \node[v] (b5) at (-1.2,0.4) {};
        \node[v] (b6) at (0,2.15) {};
        \draw[blue, line width=0.9pt] (b1)--(b2)--(b3)--(b4)--(b5)--(b1);
        \draw[blue, line width=0.9pt] (b6)--(b5);
        \draw[blue, line width=0.9pt] (b6)--(b1);
    \end{tikzpicture}

    \caption{From left to right: pan, diamond and cap.}
    \label{fig:specialgraphs}
\end{figure}

A \emph{cap} is obtained from a hole by adding a vertex adjacent to exactly two consecutive vertices of the hole; see Figure~\ref{fig:specialgraphs}. Cameron, da Silva, Huang, and Vu\v{s}kovi\'c~\cite{CSHV18} proved that every $\{\mathrm{cap},\mathrm{even\ hole}\}$-free graph $G$ satisfies $\chi(G)\leq \left\lceil\frac{3}{2}\omega(G)\right\rceil$, and asked whether the coefficient $\frac{3}{2}$ could be improved to $\frac{5}{4}$. Chen, Xu, and Xu~\cite{CXX} answered this question affirmatively by proving that every $\{\mathrm{cap},\mathrm{even\ hole}\}$-free graph $G$ satisfies $\chi(G)\leq \left\lceil\frac{5}{4}\omega(G)\right\rceil$. They further proved that the stronger bound $\chi(G)\leq \left\lceil\frac{7}{6}\omega(G)\right\rceil$ holds when $5$-holes are also excluded. The two bounds above correspond to the first two cases of a natural generalization. Motivated by this observation, Chen, Xu, and Xu~\cite{CXX} proposed the following problem.

\begin{problem}[Chen--Xu--Xu~\cite{CXX}]\label{problem:cxx}
Let $q\geq 3$ be an integer. Does every $\{\mathrm{cap},\mathrm{even\ hole}\}$-free graph $G$ with no odd hole of length at most $2q-1$ satisfy

$$
\chi(G)\leq \left\lceil\frac{2q+1}{2q}\omega(G)\right\rceil?
$$
\end{problem}

In this paper, we answer Problem~\ref{problem:cxx} affirmatively. 

\begin{theorem}\label{thm:main}
Let $q\geq 3$ be an integer. If $G$ is a $\{\mathrm{cap},\mathrm{even\ hole}\}$-free graph with no odd hole of length at most $2q-1$, then

$$
\chi(G)\leq \left\lceil\frac{2q+1}{2q}\omega(G)\right\rceil.
$$
Moreover, the bound is sharp for every $q\geq 3$.
\end{theorem}
The remainder of the paper is organized as follows. Section~\ref{sec:prelim} collects the structural results and auxiliary lemmas used in the proof. Section~\ref{sec:ear-coloring} establishes the coloring extension lemmas for clique blowups of paths and clique blowups of two odd holes sharing an edge. 
We show Theorem~\ref{thm:main} and give a sharp construction in Section~\ref{sec:main-proof}.

\section{Preliminaries}\label{sec:prelim}

In this section, we collect the notation and auxiliary results used in
the proof of Theorem~\ref{thm:main}. A \emph{clique cutset} of a graph
$G$ is a clique $K$ such that $G-K$ has more components than $G$. A \emph{universal clique} of $G$ is a clique $K$ (possibly empty) that is complete to $V(G)\setminus K$.
 A graph $G$ is \emph{odd-signable} if
there exists a function $w:E(G)\to\{0,1\}$ such that
$\sum_{e\in E(C)}w(e)$ is odd for every induced cycle $C$ of $G$.
For two disjoint vertex sets $A$ and $B$, we say that $A$ is
\emph{complete to} $B$ if every vertex of $A$ is adjacent to every
vertex of $B$. We say that $A$ is \emph{anticomplete to} $B$ if no
vertex of $A$ is adjacent to a vertex of $B$.

The following structural result is important for the structure of
$\{\mathrm{cap},4\text{-}\mathrm{hole}\}$-free graphs.
It says that, after choosing a maximal triangle-free subgraph $F$, the
whole graph is obtained from $F$ by replacing vertices with cliques,
apart from a possible universal clique. We shall use it in
Corollary~\ref{cor:structural-reduction} to reduce the coloring
problem to a clique blowup of a triangle-free graph.

\begin{lemma}[Cameron--da Silva--Huang--Vu\v{s}kovi\'c~\cite{CSHV18}]\label{lem:cameron}
Let $G$ be a $\{\mathrm{cap},4\text{-}\mathrm{hole}\}$-free graph that
contains a hole and has no clique cutset. Let $F$ be a maximal
connected induced subgraph of $G$ with at least three vertices such
that $F$ is triangle-free and has no clique cutset. Then $G$ is
obtained from $F$ by replacing each vertex of $F$ with a nonempty
clique and then adding a universal clique.
\end{lemma}

We also need the following notions. A \emph{wheel} in a graph $G$ is a
pair $(H,v)$, where $H$ is a hole of $G$ and $v\notin V(H)$ has at
least three neighbors on $H$. The vertex $v$ is called the
\emph{center} of the wheel.

Let $F$ be an induced subgraph of a graph $G$. An \emph{ear} with
attachments $x,z\in V(F)$ is an induced $xz$-path whose internal
vertices lie in $V(G)\setminus V(F)$. We say that $G$ is obtained from
$F$ by an \emph{ear addition} if there exist three consecutive
vertices $x,y,z$ of a hole in $F$ and an ear $P$ with attachments $x$
and $z$ such that the internal vertices of $P$ are precisely the
vertices of $V(G)\setminus V(F)$, and each internal vertex of $P$ has
no neighbor in $V(F)\setminus\{x,y,z\}$. Such an ear addition is called
\emph{good} if
\begin{itemize}
    \item $y$ has an odd number of neighbors in $V(P)$;
    \item $F$ contains no wheel $(H_1,v)$ such that
    $x,y,z\in V(H_1)$ and $vy\in E(G)$; and
    \item $F$ contains no wheel $(H_2,y)$ such that $x$ and $z$ are
    both neighbors of $y$ on $H_2$.
\end{itemize}
A \emph{good ear decomposition} of a graph $F$ is a sequence
$F_0,F_1,\ldots,F_d=F$ such that $F_0$ is a hole and, for each
$1\leq i\leq d$, $F_i$ is obtained from $F_{i-1}$ by a good ear
addition.
We call $F_0$ \emph{base graph}.

The graph obtained from the complete bipartite graph $K_{4,4}$ by
deleting a perfect matching is called the \emph{cube}.
The next lemma is used after the clique-blowup reduction. It gives
a good ear decomposition for the triangle-free graph $F$ obtained
from Lemma~\ref{lem:cameron}. In the proof of the main theorem, we
will use the last ear in this decomposition.

\begin{lemma}[Conforti--Cornu\'ejols--Kapoor--Vu\v{s}kovi\'c ~\cite{CCKV2000}]\label{lem:ear-decomposition}
Let $G$ be a connected triangle-free graph with at least three
vertices. Suppose that $G$ is not the cube and has no clique cutset.
Then $G$ is odd-signable if and only if it has a good ear
decomposition.
\end{lemma}

Substituting a vertex $v$ of a graph $G$ by a graph $H$ is the
operation that produces a graph with vertex set
$V(H)\cup V(G-v)$ and edge set
$$
E(H)\cup E(G-v)\cup
\{xy\mid x\in V(H),\ y\in N_G(v)\}.
$$
When $H$ is a clique, not necessarily nonempty, this operation is
called \emph{blowing up} $v$ into a clique. The clique is called the \emph{bag} of $v$. A graph obtained from a
graph $H$ by blowing up all its vertices into cliques is called a
\emph{clique blowup} of $H$. If every such clique is nonempty, then the
resulting graph is called a \emph{nonempty clique blowup} of $H$. If
every such clique has size $t$, then the resulting graph is called the
\emph{$t$-clique blowup} of $H$ and is denoted by $H^t$.
The following result treats the case where the graph $F$ in the
clique blowup is an odd cycle.
It will be used in the proof of Theorem~\ref{thm:main} for the case when
the good ear decomposition has no added ear. It will also be used
for the sharpness construction.

\begin{lemma}[Chen--Xu~\cite{ChenXu2024}]
\label{lem:odd-cycle-blowup}
Let $\ell\geq2$ be an integer, and let $G$ be a clique blowup of
$C_{2\ell+1}$. Then
$$
\chi(G)\leq
\left\lceil\frac{2\ell+1}{2\ell}\omega(G)\right\rceil.
$$
Moreover, this bound is attained by $C_{2\ell+1}^k$ for every positive
integer $k$.
\end{lemma}

The next corollary is the main structural theorem in this section.
It reduces the graph $G$ to a nonempty
clique blowup of a triangle-free graph $F$. 

\begin{corollary}\label{cor:structural-reduction}
Let $G$ be a connected $\{\mathrm{cap},\mathrm{even\ hole}\}$-free
graph that contains a hole and has neither a clique cutset nor a
universal clique. Then $G$ is a nonempty clique blowup of a
triangle-free graph $F$ with no clique cutset. Moreover, $F$ admits a
good ear decomposition whose base graph is an odd hole.
\end{corollary}

\begin{proof}
Let $C$ be a hole of $G$. Since $G$ is even-hole-free, $C$ is an odd
hole. In particular, $C$ is connected, triangle-free, and has no
clique cutset. Choose a maximal connected induced subgraph $F$ of $G$
containing $C$ such that $F$ is triangle-free and has no clique
cutset.

By Lemma~\ref{lem:cameron}, $G$ is obtained from $F$ by blowing up each
vertex of $F$ into a nonempty clique and then adding a universal
clique. Since $G$ has no universal clique, the latter clique is empty.
Thus $G$ is a nonempty clique blowup of $F$.

Since $F$ is  $\{\mathrm{triangle},\mathrm{even\ hole}\}$-free, assigning weight one to
every edge shows that $F$ is odd-signable. Moreover, $F$ is not the
cube, because the cube contains a $4$-hole. Hence
Lemma~\ref{lem:ear-decomposition} implies that $F$ has a good ear
decomposition. Its initial hole is odd because $F$ is $\{\mathrm{triangle},\mathrm{even\ hole}\}$-free.
This completes the proof of Corollary~\ref{cor:structural-reduction}.
\end{proof}

\section{Extending a coloring through the last ear}\label{sec:ear-coloring}

In this section, we prove the coloring extension lemmas needed for the
last ear. We first consider a clique blowup of a path with prescribed
color sets on its end bags.

For the next lemma, let $\mathcal{P}$ be a clique blowup of the path
$u_1v_1u_2v_2\ldots u_rv_r$, where $r\geq2$. For each
$1\leq i\leq r$, let $U_i$ and $V_i$ be the bags corresponding to
$u_i$ and $v_i$, respectively. Thus the bags of $\mathcal{P}$ occur in
the order
$U_1,V_1,U_2,V_2,\ldots,U_r,V_r$.

\begin{lemma}\label{lem:path-extension}
Let $H$ be obtained from $\mathcal P$ by adding a nonempty clique $Y$
complete to $U_1\cup V_r$ and anticomplete to every other bag of
$\mathcal P$. Suppose that $\omega(H)\leq\omega$, and let
$k=\omega+s$, where $s\geq1$. Suppose that $Y\cup U_1$ is already
properly colored with colors from $[k]$. Let $A$ and $C$ be the color
sets used on $Y$ and $U_1$, respectively, and let
$B\subseteq[k]\setminus A$ with $|B|=|V_r|$. If
$$
|B\cap C|\leq
\sum_{i=2}^r\bigl(k-|V_{i-1}|-|U_i|\bigr),
$$
then the coloring of $Y\cup U_1$ extends to a $k$-coloring of $H$ in
which $V_r$ uses the color set $B$.
\end{lemma}

\begin{proof}
For each $1\leq i\leq r$, we shall choose a color set $C_i$ for
$V_i$. Set
$
a_1=\min\{|V_1|,|B\setminus C|\}.
$
We first choose a set $C_1\subseteq[k]\setminus C$ of size $|V_1|$
such that $|C_1\cap B|=a_1$. Such a choice is possible. Indeed, if
$a_1=|V_1|$, we use $|V_1|$ colors from $B\setminus C$. Otherwise,
$a_1=|B\setminus C|$, and the number of further colors required is
$|V_1|-|B\setminus C|$. Since $|C|=|U_1|$ and $U_1\cup V_1$ is a
clique,
\begin{align*}
|[k]\setminus(B\cup C)|
 -\bigl(|V_1|-|B\setminus C|\bigr)
&=k-|C|-|V_1|\geq0.
\end{align*}
Thus the remaining colors can be chosen from
$[k]\setminus(B\cup C)$.

For $2\leq i\leq r$, define recursively
$
a_i=\min\bigl\{|V_i|,|B|,
 a_{i-1}+k-|V_{i-1}|-|U_i|\bigr\}.
$
\begin{claim}\label{a_i}
    $a_i\geq |V_i|+|B|-k $, for every $1\leq i\leq r$.
\end{claim}
\begin{subproof}[\bf Proof of Claim~\ref{a_i}]
     For $i=1$, both $|V_1|$ and
$|B\setminus C|$ are at least $|V_1|+|B|-k$: the latter follows from
$|B\cap C|\leq |C|=|U_1|\leq k-|V_1|.
$
Suppose that $i\geq2$ and that Claim~\ref{a_i} holds for $i-1$. Clearly, $\min\{|V_i|,|B|\}\geq |V_i|+|B|-k$. Therefore, it suffices to prove that $a_i=a_{i-1}+k-|V_{i-1}|-|U_i|$. Since
$U_i\cup V_i$ is a clique,
\begin{align*}
a_{i-1}+k-|V_{i-1}|-|U_i|
\geq |V_{i-1}|+|B|-k
   +k-|V_{i-1}|-|U_i|
=|B|-|U_i|
\geq |V_i|+|B|-k.
\end{align*}
 This proves Claim~\ref{a_i}.
\end{subproof}

It follows from Claim~\ref{a_i} that, for each $2\leq i\leq r$, there exists a set
$C_i\subseteq[k]$ of size $|V_i|$ such that $|C_i\cap B|=a_i$.
Choose $C_i\subseteq[k]$ with $|C_i|=|V_i|$ and
$|C_i\cap B|=a_i$. We choose its colors in $B$ and outside $B$
separately. 
First keep $\min\{a_i,a_{i-1}\}$ colors from $C_{i-1}\cap B$; if this is fewer than $a_i$, add $a_i-\min\{a_i,a_{i-1}\}$ colors from $B\setminus C_{i-1}$. 
Then keep as many colors as possible from $C_{i-1}\setminus B$, but no more than $|V_i|-a_i$ colors, and fill the remaining places with colors from $[k]\setminus(B\cup C_{i-1})$. The resulting set is $C_i$, and it satisfies $|C_i|=|V_i|$ and $|C_i\cap B|=a_i$.
Thus $|C_i\cap C_{i-1}|$ is as large as
possible among all sets with the prescribed size and prescribed
intersection with $B$.

\begin{claim}\label{claim:last-bag}
$C_r=B$.
\end{claim}
\begin{subproof}[\bf Proof of Claim~\ref{claim:last-bag}]
     Suppose otherwise. Since $|C_r|=|V_r|=|B|$, we have
$a_r<|B|=|V_r|$. For $2\leq i\leq r$, set
$
d_i=k-|V_{i-1}|-|U_i|.
$
Since $V_{i-1}\cup U_i$ is a clique and $k=\omega+s$, we have
$d_i\geq s>0$.

First suppose that there is an index $j$ with $a_j=|V_j|$, and choose
$j$ largest. Then $j<r$. We have $a_i<|V_i|$ for every $i>j$. We also
have $a_i<|B|$ for every $i>j$. Indeed, if $a_h=|B|$ for some $h>j$,
let $h'>h$ be the first index with $a_{h'}<|B|$. Since
$d_{h'}>0$, by the definition of $a_{h'}$,
we have
$a_{h'}=|V_{h'}|$ and this contrary to the choice of $j$. Therefore, for every
$j<i\leq r$, $a_i=a_{i-1}+k-|V_{i-1}|-|U_i|=a_{i-1}+d_i$. It follows that
\begin{align*}
a_r
=|V_j|+\sum_{i=j+1}^r
  \bigl(k-|V_{i-1}|-|U_i|\bigr)
=|V_r|+\sum_{i=j+1}^r
  \bigl(k-|V_i|-|U_i|\bigr)
\geq |V_r|=|B|,
\end{align*}
a
contradiction.

It remains to consider the case in which $a_i<|V_i|$ for every $i$.
As above, $a_i<|B|$ for every $i$, for otherwise the first later drop
below $|B|$ would force equality with the size of the corresponding
bag. Consequently,
$a_1=|B\setminus C|$ and $a_i=a_{i-1}+d_i$ for every
$2\leq i\leq r$. By the hypothesis of the lemma,
\begin{align*}
a_r
=|B\setminus C|+
  \sum_{i=2}^r\bigl(k-|V_{i-1}|-|U_i|\bigr)
\geq |B\setminus C|+|B\cap C|
=|B|,
\end{align*}
again a contradiction. This proves Claim~\ref{claim:last-bag}.
\end{subproof}

\begin{claim}\label{claim:consecutive-color-sets}
For each $2\leq i\leq r$,
$|C_{i-1}\cup C_i|\leq k-|U_i|$.
\end{claim}
\begin{subproof}[\bf Proof of Claim~\ref{claim:consecutive-color-sets}]
Fix $2\leq i\leq r$. Suppose first that $a_i<a_{i-1}$. Since
$|B|\geq a_{i-1}$ and
$a_{i-1}+k-|V_{i-1}|-|U_i|>a_{i-1}$, the definition of $a_i$ implies
that $a_i=|V_i|$. Thus $C_i\subseteq B$. By the choice of $C_i$,
$C_i\subseteq C_{i-1}\cap B$. Hence
$
|C_{i-1}\cup C_i|=|V_{i-1}|\leq k-|U_i|.
$

Now suppose that $a_i\geq a_{i-1}$. By the choice of $C_i$, it uses all colors in
$C_{i-1}\cap B$ and then as many colors as possible from
$C_{i-1}\setminus B$. If the colors in $C_{i-1}\setminus B$ are
enough to supply all $|V_i|-a_i$ colors outside $B$, then
$
|C_i\setminus C_{i-1}|=a_i-a_{i-1}
\leq k-|V_{i-1}|-|U_i|,
$
and therefore
$
|C_{i-1}\cup C_i|
\leq |V_{i-1}|+k-|V_{i-1}|-|U_i|
=k-|U_i|.
$
If these colors are not enough, then every color of $C_{i-1}$ is used
in $C_i$, and hence $C_{i-1}\subseteq C_i$. Therefore,
$$
|C_{i-1}\cup C_i|=|C_i|=|V_i|\leq k-|U_i|.
$$
This proves Claim~\ref{claim:consecutive-color-sets}.
\end{subproof}

\medskip

By Claim~\ref{claim:last-bag}, $C_r=B$.
Then by Claim~\ref{claim:consecutive-color-sets},
for each $2\leq i\leq r$,
there are at least
$|U_i|$ colors outside $C_{i-1}\cup C_i$; 
we may use any $|U_i|$ of them on
$U_i$. Together with the given coloring of $Y\cup U_1$, the chosen
color sets give a proper $k$-coloring of $H$. This completes the proof
of Lemma~\ref{lem:path-extension}.
\end{proof}

\begin{lemma}\label{lem:two-odd-holes}
Let $q\geq3$ be an integer, and let $H$ be a clique blowup of the union
of two odd holes, each of length at least $2q+1$, that intersect
exactly in an edge $yz$. Let $Y$ be the bag corresponding to $y$, and
let $X$ and $X'$ be the bags corresponding to the two neighbors of
$y$ other than $z$. Let $w\geq\omega(H)$ be an integer, and set
$k=w+\lceil w/(2q)\rceil$. Then every $k$-coloring of
$H[X\cup Y\cup X']$ extends to a $k$-coloring of $H$.
\end{lemma}

\begin{proof}
Let $s=\lceil w/(2q)\rceil$, so $k=w+s$ and $w\leq2qs$. Write the two
holes as $yx_0x_1\ldots x_{2\ell-1}y$ and $yx_0'x_1'\ldots x_{2t-1}'y$, where $x_{2\ell-1}=x_{2t-1}'=z$ and
$\min\{\ell,t\}\geq q$. Let $X_i$ and $X_i'$ be the bags
corresponding to $x_i$ and $x_i'$, respectively, and set
$Z=X_{2\ell-1}=X_{2t-1}'$. Thus $X_0=X$ and $X_0'=X'$.

Fix the given $k$-coloring of $H[Y\cup X_0\cup X_0']$. Let $A$, $C_0$,
and $C_0'$ be the color sets used on $Y$, $X_0$, and $X_0'$, respectively,
and let $U=[k]\setminus A$. Since $Y$ is complete to both $X_0$ and
$X_0'$, we have $C_0,C_0'\subseteq U$.

Set
$S_0=\sum_{i=1}^{\ell-1}
\bigl(k-|X_{2i-1}|-|X_{2i}|\bigr)$
and $S_0'=\sum_{i=1}^{t-1}
\bigl(k-|X_{2i-1}'|-|X_{2i}'|\bigr)$. For
$1\leq i\leq\ell-1$, both
$X_{2i-1}\cup X_{2i}$ and $X_{2i}\cup X_{2i+1}$ are cliques. Hence $k-|X_{2i-1}|-|X_{2i}|\geq s$
and
$$
k-|X_{2i-1}|-|X_{2i}|
\geq s+|X_{2i+1}|-|X_{2i-1}|.
$$
Consequently,
$S_0\geq(q-1)s$
and
\begin{align*}
S_0
\geq(\ell-1)s+
 \sum_{i=1}^{\ell-1}
 \bigl(|X_{2i+1}|-|X_{2i-1}|\bigr)
=(\ell-1)s+|Z|-|X_1|
\geq(q-1)s+|Z|+|X_0|-w,
\end{align*}
where the last inequality follows from
$|X_0|+|X_1|\leq w$. Similarly, $S_0'\geq(q-1)s$ and $S_0'\geq(q-1)s+|Z|+|X_0'|-w$.

\begin{claim}\label{claim:common-color-set}
There is a set $B\subseteq U$ with $|B|=|Z|$ such that
$|B\cap C_0|\leq S_0$ and $|B\cap C_0'|\leq S_0'$.
\end{claim}
\begin{subproof}[\bf Proof of Claim~\ref{claim:common-color-set}]
      Let
$\delta_0=\max\{0,|X_0|-S_0\}$ and $\delta_0'=\max\{0,|X_0'|-S_0'\}$.
We shall choose a set $D\subseteq U$ of size
$k-|Y|-|Z|$ such that $|D\cap C_0|\geq\delta_0$ and $|D\cap C_0'|\geq\delta_0'$. Then $B=U\setminus D$ has the required properties.
The bounds on $S_0$ give $\delta_0\leq\max\{0,|X_0|-(q-1)s\}$
and
$\delta_0\leq\max\{0,w-|Z|-(q-1)s\}$.
The analogous inequalities hold for $\delta_0'$.

We first show that $\max\{\delta_0,\delta_0'\}\leq k-|Y|-|Z|$. If $|Z|\leq qs$, then
\begin{align*}
k-|Y|-|Z|-\bigl(|X_0|-(q-1)s\bigr)
=w-|Y|-|X_0|+qs-|Z|
\geq0,
\end{align*}
since $Y\cup X_0$ is a clique. The same argument applies to $X_0'$.
If $|Z|>qs$, then $|Y|<qs$, because $Y\cup Z$ is a clique and
$w\leq2qs$. Hence
$$k-|Y|-|Z|-\bigl(w-|Z|-(q-1)s\bigr)=qs-|Y|>0.$$
This proves $\max\{\delta_0,\delta_0'\}\leq k-|Y|-|Z|$.

We next show that $\delta_0+\delta_0'
\leq k-|Y|-|Z|+|C_0\cap C_0'|$. If one of $\delta_0,\delta_0'$ is zero, this follows immediately from
$\max\{\delta_0,\delta_0'\}\leq k-|Y|-|Z|$. Suppose that both are positive, we have $w-|Z|-(q-1)s>0$.
Therefore
$|Z|<w-(q-1)s\leq(q+1)s\leq2(q-1)s$,
where the last inequality follows from $q\geq3$. Since
$C_0,C_0'\subseteq U$,
$
|X_0|+|X_0'|-|C_0\cap C_0'|
=|C_0\cup C_0'|\leq|U|=k-|Y|.
$
Therefore, 
\begin{align*}
\delta_0+\delta_0'
\leq |X_0|+|X_0'|-2(q-1)s
<k-|Y|-|Z|+|C_0\cap C_0'|.
\end{align*}


It remains to choose $D$. If
$|C_0\cap C_0'|\geq\min\{\delta_0,\delta_0'\}$, assume by symmetry
that $\delta_0\geq\delta_0'$. 
Choose $\delta_0'$ colors from $C_0\cap C_0'$, and then choose
$\delta_0-\delta_0'$ further colors from $C_0$ that have not
already been chosen.
If
$|C_0\cap C_0'|<\min\{\delta_0,\delta_0'\}$, choose all colors in
$C_0\cap C_0'$, together with
$\delta_0-|C_0\cap C_0'|$ colors from $C_0\setminus C_0'$ and
$\delta_0'-|C_0\cap C_0'|$ colors from $C_0'\setminus C_0$.
In the first case, at most $\max\{\delta_0,\delta_0'\}$ colors are
chosen, and in the second case, at most
$\delta_0+\delta_0'-|C_0\cap C_0'|$ colors are chosen. 
By $\max\{\delta_0,\delta_0'\}\leq k-|Y|-|Z|$,
both numbers are at most $k-|Y|-|Z|$.

Since $Y\cup Z$ is a clique,
$0\leq k-|Y|-|Z|\leq|U|$. We may therefore extend the chosen set to a
set $D\subseteq U$ of size $k-|Y|-|Z|$. Let $B=U\setminus D$. Then
$|B|=|Z|$, and
\[
|B\cap C_0|
=|C_0|-|D\cap C_0|
\leq |X_0|-\delta_0
\leq S_0.
\]
Similarly, $|B\cap C_0'|\leq S_0'$. This proves
Claim~\ref{claim:common-color-set}.
\end{subproof}
  
\medskip

Apply Lemma~\ref{lem:path-extension} to the path with bags
$Y,X_0,X_1,\ldots,X_{2\ell-1}=Z$ by setting
$U_i=X_{2i-2}$ and $V_i=X_{2i-1}$ for $1\le i\le \ell$,
and 
prescribing the color set $B$ on $Z$. 
We verify that $|B\cap C_0|\leq S_0$. 
Similarly, we
apply the same lemma to the path with bags
$Y,X_0',X_1',\ldots,X_{2t-1}'=Z$ with
$U_i=X_{2i-2}'$ and $V_i=X_{2i-1}'$ for $1\le i\le t$,
and 
again prescribing $B$ on $Z$. This is possible because
$|B\cap C_0'|\leq S_0'$.

Since $Y\cup Z$ is a clique and the set
$X_0\cup\cdots\cup X_{2\ell-2}$ is anticomplete to
$X_0'\cup\cdots\cup X_{2t-2}'$,
the
two colorings together give a $k$-coloring of $H$ extending the given
coloring. This completes the proof of Lemma~\ref{lem:two-odd-holes}.
\end{proof}

\section{Proof of Theorem~\ref{thm:main}}\label{sec:main-proof}

\begin{proof}[Proof of Theorem~\ref{thm:main}]
Fix $q\geq3$. We first prove the upper bound. Suppose otherwise, and
let $G$ be a counterexample with $|V(G)|$ minimum. Write $\omega=\omega(G)$ and $s=\left\lceil\frac{\omega}{2q}\right\rceil$. Then $k=\omega+s
 =\left\lceil\frac{2q+1}{2q}\omega\right\rceil$.
Every proper induced subgraph of $G$ satisfies the assumptions of the
theorem and is therefore $k$-colorable by the minimality of $G$.
The graph $G$ is connected. Indeed, otherwise every component of $G$
is a proper induced subgraph and hence is $k$-colorable, and these
colorings together give a $k$-coloring of $G$.

We next prove that $G$ has no clique cutset. Suppose that $K$ is a
clique cutset of $G$, and let $D_1,\ldots,D_t$ be the components of
$G-K$, where $t\ge2$. For $1\le i\le t$, let
$G_i=G[K\cup D_i]$. Each $G_i$ is a proper induced subgraph of $G$,
and so $G_i$ has a $k$-coloring $f_i$.
Since $K$ is a clique, after
renaming colors in $f_i$ for each $i\ge2$, we may arrange that
$f_i(v)=f_1(v)$ for every $v\in K$.
Now define a coloring $f$ of $G$ as follows. For $v\in K$, set
$f(v)=f_1(v)$. For $v\in D_i$, set $f(v)=f_i(v)$. It is proper on
each graph $G_i$. Then $f$ is a $k$-coloring of $G$, a contradiction.

The graph $G$ has no universal vertex. Indeed, suppose that $v$ is
universal. Then $\omega(G-v)=\omega-1$, and minimality gives
$
\chi(G-v)
\leq \omega-1+
\left\lceil\frac{\omega-1}{2q}\right\rceil
\leq k-1.
$
Assigning a new color to $v$ gives a $k$-coloring of $G$, a
contradiction. Consequently, $G$ has no universal clique.

If $G$ has no hole, then $G$ is chordal and hence perfect, contrary to
$\chi(G)>k$. By Corollary~\ref{cor:structural-reduction}, $G$ is a
nonempty clique blowup of a triangle-free graph $F$ with a good ear
decomposition $F_0,F_1,\ldots,F_d=F$,
where $F_0$ is an odd hole.

Suppose first that $d=0$. Then $F=C_{2r+1}$ for some $r\geq2$.
Choosing one vertex from each nonempty bag gives an induced
$(2r+1)$-hole in $G$, and hence $r\geq q$. By
Lemma~\ref{lem:odd-cycle-blowup},
$$
\chi(G)
\leq\left\lceil\frac{2r+1}{2r}\omega\right\rceil
\leq\left\lceil\frac{2q+1}{2q}\omega\right\rceil
=k,
$$
a contradiction. Thus $d\geq1$.

Let $F_d$ be obtained from $F_{d-1}$ by adding the last good ear
$P=x_0x_1\ldots x_m$, and let $x_0,y,x_m$ be the three consecutive
vertices of the hole of $F_{d-1}$ to which $P$ is attached. The
vertex $y$ is adjacent to $x_0$ and $x_m$ and has an odd number of
neighbors in $V(P)$. It therefore has an internal neighbor on $P$.
Let $x_a$ and $x_b$ be the first two neighbors of $y$ on $P$ after
$x_0$. By their choice, $y$ has no neighbor in
$\{x_1,\ldots,x_{a-1}\}$ and no neighbor in
$\{x_{a+1},\ldots,x_{b-1}\}$. Since $P$ is induced, $yx_0x_1\ldots x_ay$ and $yx_ax_{a+1}\ldots x_by$
are induced cycles. Since $F$ is $\{$triangle, even-hole$\}$-free,
both cycles are odd holes. Hence $a=2\ell-1$ and $b-a=2t-1$
for some integers $\ell,t\geq2$. These holes have lengths
$2\ell+1$ and $2t+1$, respectively. Since $F$ is an induced subgraph
of $G$ and $G$ has no odd hole of length at most $2q-1$, we have
$\min\{\ell,t\}\geq q$.

Let $Y$ be the bag corresponding to $y$, and let $X_i$ be the bag
corresponding to $x_i$ for $0\leq i\leq b$. By the choice of $x_a$ and
$x_b$ and the fact that $P$ is induced, the graph induced by
$\{y,x_0,\ldots,x_b\}$ is exactly the union of the two odd holes above,
and their intersection is the edge $yx_a$. Therefore,
$
H=G[Y\cup X_0\cup\cdots\cup X_b]
$
is a clique blowup of two odd holes, each of length at least $2q+1$,
that intersect exactly in the edge corresponding to $yx_a$.
Let $W=X_1\cup\cdots\cup X_{b-1}$.
By minimality, $G-W$ has a $k$-coloring.
By Lemma~\ref{lem:two-odd-holes} applied to the subgraph $H$, we may extend the colouring on $Y\cup X_0\cup X_b$
to a $k$-colouring of $H$.

We claim that $W$ is anticomplete to $V(G)\setminus V(H)$. 
For $F$, every vertex $x_i$ with $1\leq i\leq b-1$ is an internal
vertex of the last ear and hence has no neighbor in
$V(F_{d-1})\setminus\{x_0,y,x_m\}$. Moreover, if $j\geq b+1$, then
$j-i\geq2$, so $x_ix_j\notin E(F)$ because $P$ is induced. If
$b<m$, this also excludes an edge from $x_i$ to $x_m$, since then
$i\leq b-1\leq m-2$; if $b=m$, the bag corresponding to $x_m$ already
lies in $H$. Thus no vertex corresponding to a bag in $W$ is adjacent
in $F$ to a vertex corresponding to a bag outside $H$. The same is
therefore true in the clique blowup $G$.

The colorings of $G-W$ and $H$ agree on their intersection
$Y\cup X_0\cup X_b$, and there are no edges between $W$ and
$V(G)\setminus V(H)$. Hence they combine to a $k$-coloring of $G$, a
contradiction. This proves the upper bound.

For sharpness, let $p$ be a positive integer, and let $C_{2q+1}^p$ be the
uniform clique blowup of $C_{2q+1}$ in which every bag has size $p$.
A hole of $C_{2q+1}^p$ contains at most one vertex from each bag. Indeed,
if it contained two vertices from the same bag, then these two
adjacent vertices would be consecutive on the hole and, being true
twins, would create a chord. 
Hence every hole of
$C_{2q+1}^p$ has length $2q+1$.

Every vertex outside such a hole has exactly three neighbors on it:
the selected vertex in its own bag and the selected vertices in the
two neighboring bags. Thus $C_{2q+1}^p$ is cap-free. It is also
even-hole-free and has no odd hole of length at most $2q-1$, so it
satisfies the assumptions of the theorem.

Finally, $\omega(C_{2q+1}^p)=2p$, and
Lemma~\ref{lem:odd-cycle-blowup} gives
$$
\chi(C_{2q+1}^p)
=\left\lceil\frac{(2q+1)p}{q}\right\rceil
=\left\lceil\frac{2q+1}{2q}\omega(C_{2q+1}^p)\right\rceil.
$$
Thus the bound is sharp for every $q\geq3$. This completes the proof
of Theorem~\ref{thm:main}.
\end{proof}

\section*{Acknowledgements}

This work was supported by the National Key R\&D Program of China
(No.~2022YFA1006400) and the National Natural Science Foundation of China
(No.~12571376).

\section*{Declaration}
\noindent$\textbf{Conflict~of~interest}$
The authors declare that they have no known competing financial interests or personal relationships that could have appeared to influence the work reported in this paper.
	
\noindent$\textbf{Data~availability}$
Data sharing not applicable to this paper as no datasets were generated or analysed during the current study.

\end{document}